\documentclass[twoside, 12pt]{article}
\usepackage{mathrsfs}
\usepackage{amssymb, amsmath, mathrsfs, amsthm}
\usepackage{graphicx}
\usepackage{color}
\usepackage[top=2.5cm, bottom=2.5cm, left=2.5cm, right=2.5cm]{geometry}
\usepackage{float, caption, subcaption}

\input{epsf}

\newcommand{\bd}{\begin{description}}
\newcommand{\ed}{\end{description}}
\newcommand{\bi}{\begin{itemize}}
\newcommand{\ei}{\end{itemize}}
\newcommand{\be}{\begin{enumerate}}
\newcommand{\ee}{\end{enumerate}}
\newcommand{\beq}{\begin{equation}}
\newcommand{\eeq}{\end{equation}}
\newcommand{\beqs}{\begin{eqnarray*}}
\newcommand{\eeqs}{\end{eqnarray*}}

\newcommand{\Rmnum}[1]{\expandafter\@slowromancap\romannumeral #1@}

\definecolor{DarkGreen}{rgb}{0.2, 0.6, 0.3}

\newtheorem{theorem}{Theorem}[section]

\newtheorem{lemma}{Lemma}[section]
\newtheorem{definition}{Definition}

\newtheorem{claim}{Claim}
\newtheorem{remark}{Remark}[section]
\newtheorem{example}{Example}[section]

\begin{document}
\title{Transversals in a collections of trees}
\author{\small Ethan Y.H. Li$^a$, Luyi Li$^{b,c}$, Ping Li$^a$\\
{\small $^a$School of Mathematics and Statistics}\\
{\small  Shaanxi Normal University, Xi'an, Shaanxi 710062, China}\\
{\small $^b$Center for Combinatorics and LPMC}\\
{\small Nankai University, Tianjin 300071, China}\\
{\small $^c$Laboratoire Interdisciplinaire des Sciences du Num\'erique}\\
{\small CNRS-Universit\'e Paris-Saclay, Orsay 91405, France}\\
{\small Emails: yinhao\_li@snnu.edu.cn, liluyi@mail.nankai.edu.cn, lp-math@snnu.edu.cn}
}

\date{}

\maketitle

\begin{abstract}

Let $\mathcal{S}$ be a fixed family of graphs on vertex set $V$ and $\mathcal{G}$ be a collection of elements in $\mathcal{S}$. We investigated the transversal problem of finding the maximum value of $|\mathcal{G}|$ when $\mathcal{G}$ contains no rainbow elements in $\mathcal{S}$. Specifically, we determine the exact values when $\mathcal{S}$ is a family of stars or a family of trees of the same order $n$ with $n$ dividing $|V|$. Further, all the extremal cases for $\mathcal{G}$ are characterized.
\end{abstract}

\section{Introduction}


In 1974, D\'{e}nes and Keedwell \cite{DK} conjectured that every $n\times n$ Latin square has a set of entries of order $n-1$ which contains at most one representative
of each row and column and no symbol is repeated.
Every such set is called a {\em partial transversal} of the Latin square.
A lot of scholars have made contributions to this conjecture in the last few decades, see \cite{KK,DD,BVW,W,HS}.
From another point of view, an equivalent statement of the conjecture is: for any proper edge-coloring of the balanced complete bipartite graph $K_{n,n}$ with $n$ colors, the edge-colored graph $K_{n,n}$ has a rainbow matching of size at least $n-1$.

More generally, this concept may be extended as follows: given a collection of graphs $\mathcal{G}=\{G_1,G_2,\ldots,G_t\}$ (not necessarily distinct) on vertex set $V$ and a graph $H$, $\mathcal{G}$ is said to contain a \emph{rainbow $H$} if there exists a graph isomorphic to $H$ consisting of at most one edge from each $G_i$.
We say that $\mathcal{G}$ is \emph{rainbow $H$-free} if
$\mathcal{G}$ contains no rainbow $H$. Using this concept, the well-known Rota's basis conjecture (restricted to graphic matroids) may be reformulated as: any collection $\mathcal{G}=\{G_1,G_2,\ldots,G_n\}$ of spanning trees of a graph $H$ of order $n+1$ contains $n$ disjoint rainbow spanning trees of $H$. As the first step of our approach to this conjecture, we intend to find the minimum size of $\mathcal{G}$ preserving the existence of one rainbow tree of a given order or structure. Equivalently, we aim to determine the maximum size for $\mathcal{G}$ to exclude any rainbow tree of a given order or structure, which is analogous to Ramsey and Tur\'{a}n problems.

There have been plenty of researchers interested in these problems, and they also studied rainbow graphs other than trees. For instance, Aharoni, DeVos, de la Maza, Montejano and \v{S}\'amal \cite{ADGMS}
gave a rainbow version of Mantel's theorem:
a collection $\mathcal{G}=\{G_1,G_2,G_3\}$ of $n$-graphs with $|E(G_i)|>\frac{1+\tau^2}{4}n^2$ for all $1\leq i\leq 3$ contains
a rainbow triangle,
where $\tau=\frac{4-\sqrt7}{9}$.
In 2020, Joos and Kim \cite{JK} proved a rainbow version of Dirac's theorem: if $\mathcal{G}=\{G_i: i\in [n]\}$ is a collection of not necessarily distinct $n$-graphs with the same vertex set of order $n$
and $\delta(G_i)\geq \frac{n}{2}$ for $i\in[n]$, then there exists a rainbow Hamiltonian cycle in $\mathcal{G}$.
For more results on rainbow structures in a family of graphs, please refer to \cite{Pe,BHS,CHWW,CWZ,LLL,MMP,FHM}.

In addition to the two above results which require edge or degree conditions, researchers also considered the problems of finding a rainbow graph $H$ from $\mathcal{G}$ where $H$ and elements of $\mathcal{G}$ belong to the same class.
Aharoni, Briggs, Holzman and Jiang \cite{ABHJ} proved that every family
of $2\left\lceil \frac{n}{2}\right\rceil-1$ odd cycles on $n$ vertices
contains a rainbow odd cycle.
Dong and Xu \cite{DX} showed that any collection
of $\left\lfloor\frac{6(n-1)}{5}\right\rfloor+1$ even cycles on $n$ vertices
contains a rainbow even cycle.
Moreover, Goorevitch and Holzman \cite{GH} proved that every family
of $(1+o(1))\frac{n^2}{8}$ triangles on $n$ vertices
contains a rainbow triangle.

In this paper, we continue to investigate this topic on rainbow stars and general rainbow trees, and the main results are the following two theorems.

\begin{theorem}\label{main-0}
Let $\mathcal{S}$ be a collection of stars $K_{1,\Delta}$ on vertex set $V$ with $|V|=n=a(2\Delta-1)+b$ and $b(\Delta-1)=k_1(2\Delta-1)+k_2$, where $a,b,k_1,k_2$ are nonnegative integers and $0\leq b,k_2\leq 2\Delta-2$.
Then the maximum value of $|\mathcal{S}|$ for $\mathcal{S}$ to be rainbow $K_{1,\Delta}$-free is
$$
\left\{
\begin{array}{ll}
a(\Delta-1)^2+k_1(\Delta-1), &  \mbox{if }a\geq 1\mbox{ and }0\leq k_2\leq\Delta; \vspace{0.05cm}\\

a(\Delta-1)^2+k_1(\Delta-1)+k_2-\Delta,
&  \mbox{if }a\geq 1\mbox{ and }\Delta \le k_2\leq 2\Delta-2;\\
\left\lfloor\frac{(n-1)^2}{4}\right\rfloor, & \mbox{if }a=0.\\
\end{array}
\right.$$
Moreover, the bounds are tight, and $\mathfrak{A}(n,\Delta)$ (defined in Section \ref{sec-2}) is the set of all rainbow $K_{1,\Delta}$-free collections $\mathcal{S}$ with $|\mathcal{S}|$ attaining these maximum values.
\end{theorem}

\begin{theorem}\label{main-1}
Let $\mathcal{T}=\{T_1,\ldots,T_t\}$ be a collection of trees on $V$ with $|V|=m$, $|T_i|=n$ for each $i\in[t]$ and $n|m$.
Then the value of $t$ is at most $\frac{m(n-2)}{n}$ if $\mathcal{T}$ contains no rainbow tree of order $n$.
Moreover, the bounds are tight, and $\mathfrak{B}(n,m)$ (defined in Section \ref{sec-thm2}) is the set of all such collections $\mathcal{T}$ with $|\mathcal{S}|$ attaining the maximum value.
\end{theorem}

%



The proofs will be presented in the following two sections, before which we would like to introduce some additional notation.
For a positive integer $n$ we use $[n]$ to denote the set $\{1,2,\ldots, n\}$.
For a graph $G$, we use $V(G)$ and $E(G)$ to denote the \emph{vertex set} and
\emph{edge set} of $G$, respectively.
For any two vertex sets $X$ and $Y$, we use
$X\vee Y$ to denote a graph obtained by adding an edge between each vertex of $X$ and each vertex of $Y$.
For a subset $X$ of $V$, $\partial_G(X)$ denotes the set of edges between $X$ and $V-X$ in $E(G)$.

For a digraph $\overrightarrow{D}$, we use $V(\overrightarrow{D})$ and $A(\overrightarrow{D})$ to denote \emph{vertex set} and
\emph{arc set} of $\overrightarrow{D}$, respectively.
For a subset $X$ of $V(\overrightarrow{D})$, $\overrightarrow{D}[X]$ denotes
the subdigraph of $\overrightarrow{D}$ induced by $X$.
Given a vertex $v$ in a digraph $\overrightarrow{D}$, we use $N_{\overrightarrow{D}}^+(v)$ and $N_{\overrightarrow{D}}^-(v)$ to denote the set of \emph{out-neighbours} and \emph{in-neighbours} of $v$ in $\overrightarrow{D}$, respectively.
The \emph{out-degree} (resp. \emph{in-degree}) of $v$ in $\overrightarrow{D}$,
denoted by $d_{\overrightarrow{D}}^+(v)$ (resp. $d_{\overrightarrow{D}}^-(v)$), is the number of out-neighbours (resp. in-neighbours) of $v$ in $\overrightarrow{D}$.
A digraph is \emph{$d$-out-regular} (resp. \emph{$d$-in-regular}) if the out-degrees (resp. in-degrees) of all vertices are $d$.

\section{Proof of Theorem \ref{main-0}}\label{sec-2}
This section is devoted to proving Theorem \ref{main-0}. At first, we present some basic notation and give the definition of $\mathfrak{A}(n,\Delta)$, which describes all the extremal cases of Theorem \ref{main-0}.

For a collection $\mathcal{S}$ of stars $K_{1,\Delta}$ on vertex set $V$, we denote by $\mathcal{S}_u$ the set of all stars with center $u$. We partition $V$ into the set of centers $C=\{u\in V:\mathcal{S}_u\neq \emptyset\}$ and the set the other vertices $L=V-C$, and further set $\overrightarrow{D}$ to be the digraph with $$V(\overrightarrow{D})=V \quad \mbox{and} \quad A(\overrightarrow{D})=\{(x,y):xy \mbox{ is an edge of }\bigcup_{S\in \mathcal{S}_x}E(S) \}.$$
Then it is easy to verify that between any pair of vertices in $C$ there are at most two arcs and parallel arcs will not appear (symmetric arcs may exist). In addition, there is no symmetric arc or parallel arc between $C$ and $L$, and $L$ is an independent set in $\overrightarrow{D}$.

\begin{definition}\label{def-1}
Let $V, n, \Delta, a, b,k_1,k_2$ be defined as in Theorem \ref{main-0}.
We define $\mathfrak{A}(n,\Delta)$ to be the set of collections $\mathcal{S}$ on vertex set $V$ consisting of stars $K_{1,\Delta}$ satisfying the following conditions:
\begin{enumerate}
\item [(i)] for $n \ge 2\Delta-1$ and $k_2<\Delta$, there are $a(\Delta-1)^2+k_1(\Delta-1)$ stars in $\mathcal{S}$ such that
     \begin{itemize}
       \item $|C|=a(\Delta-1)+k_1$;
       \item for each vertex $u\in C$, there are exactly $\Delta-1$ stars $K_{1,\Delta}$ with center $u$, and all leaves of these stars are in $L$;
       \item for each vertex $v\in L$, $d_{\overrightarrow{D}}^-(v)\leq \Delta-1$.
     \end{itemize}

\item [(ii)] for $n \ge 2\Delta-1$ and $k_2>\Delta$, there are $a(\Delta-1)^2+k_1(\Delta-1)+k_2-\Delta$ stars $K_{1,\Delta}$ in $\mathcal{S}$ such that
    \begin{itemize}
      \item $|C|=a(\Delta-1)+k_1+1$ and $|A(\overrightarrow{D}[C])|=2\Delta-1-k_2$;
      \item for each vertex $u\in C$, there are exactly $\Delta-1-d_{\overrightarrow{D}}^-(u)$ copies of  $K_{1,\Delta}$ with center $u$;
      \item for each vertex $v\in L$, $d_{\overrightarrow{D}}^-(v)=\Delta-1$.
    \end{itemize}

\item [(iii)] for $n \ge 2\Delta-1$ and $k_2=\Delta$, there are $a(\Delta-1)^2+k_1(\Delta-1)=a(\Delta-1)^2+k_1(\Delta-1)+k_2-\Delta$ stars in $\mathcal{S}$, and $\mathcal{S}$ satisfies conditions of $(i)$ or $(ii)$.

\item [(iv)] for $\Delta+1\leq n\leq 2\Delta-2$, there are $\left\lfloor\frac{(n-1)^2}{4}\right\rfloor$ stars $K_{1,\Delta}$ in $\mathcal{S}$ such that
\begin{itemize}
\item if $n$ is odd, then $|C|=\left\lfloor\frac{n-1}{2}\right\rfloor$; if $n$ is even, then either $|C|=\left\lfloor\frac{n-1}{2}\right\rfloor$ or $|C|=\left\lceil\frac{n-1}{2}\right\rceil$;

\item $\overrightarrow{D}[C]$ is $(\Delta-|L|)$-out-regular;

\item for each vertex $u\in C$, $\mathcal{S}_u$ consists of $\Delta-1-d_{\overrightarrow{D}}^-(u)$ copies of $K_{1,\Delta}$, and $L$ is contained in the set of leaves of each star;
\end{itemize}
\end{enumerate}

\end{definition}

Now we shall present some examples to illustrate this definition.

\begin{figure}[ht]
    \centering
    \includegraphics[width=350pt]{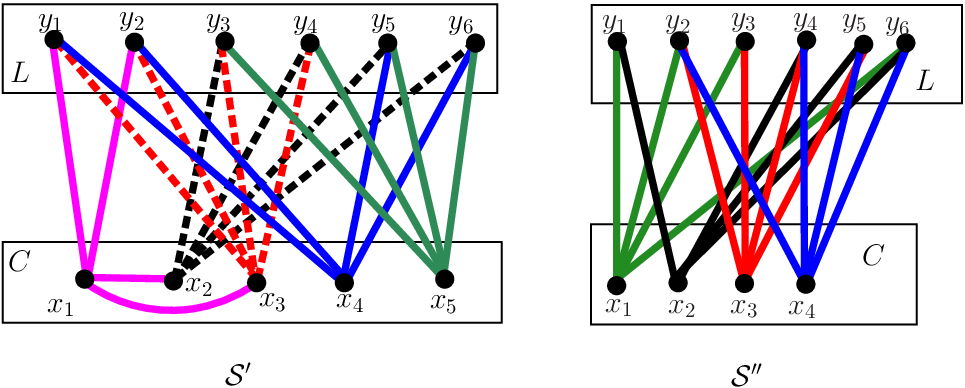}\\
    \caption{$\mathcal{S}'\in \mathfrak{A}(11,4)$ \mbox{ and } $\mathcal{S}'\in \mathfrak{A}(10,4)$.} \label{pict-2}
\end{figure}

\begin{example}
In the left part of Figure \ref{pict-2}, $\mathcal{S}'$ is an element of $\mathfrak{A}(n,\Delta)$ with $n=11$, $\Delta =4$, $a=k_1=1$, $b=4$ and $k_2=5>\Delta$.
Each solid monochromatic star represents three copies of stars, and each dashed monochromatic star represents two copies of stars.
Then $|\mathcal{S}'|=13=a(\Delta-1)^2+k_1(\Delta-1)+k_2-\Delta$ and $\overrightarrow{D}[C]$ has two arcs $(x_1,x_2)$ and $(x_1,x_3)$.
In the right part, $\mathcal{S}''$ belongs to $\mathfrak{A}(n,\Delta)$ with $n=10$, $\Delta=4$, $a=k_1=1$, $b=3$ and $k_2=2<\Delta$.
Each solid monochromatic star represents three copies of stars.
$\mathcal{S}''$ satisfies all conditions of $(i)$.
\end{example}

\begin{figure}[ht]
    \centering
    \includegraphics[width=250pt]{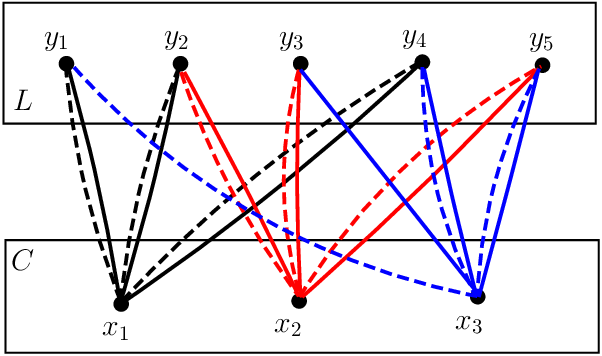}\\
    \caption{$\mathcal{S}^*\in \mathfrak{A}(8,3)$} \label{pict-1}
\end{figure}

\begin{remark}
If $n<2\Delta-1$ or $n\geq 2\Delta-1$ and $k_2> \Delta$, then for each vertex $u\in C$, $\mathcal{S}_u$ consists of copies of the same star with center $u$, that is, any two stars in $\mathcal{S}_u$ have the same leaves.
However, if $n\geq 2\Delta-1$ and $k_2< \Delta$, then
it is not necessary for
the stars in $\mathcal{S}_u$ to have the same leaves, see $\mathcal{S}^*\in \mathfrak{A}(8,3)$ of Figure \ref{pict-1} as an example. Here each solid or dashed monochromatic star represents only one star.
\end{remark}

We will confirm that the definition of $\mathfrak{A}(n,\Delta)$ is meaningful, that is, $\mathfrak{A}(n,\Delta)\neq \emptyset$ and each element of $\mathfrak{A}(n,\Delta)$ is rainbow $K_{1,\Delta}$-free.
\begin{lemma}\label{lemm}
If $\mathfrak{A}(n,\Delta)\neq \emptyset$, then each $\mathcal{S}\in \mathfrak{A}(n,\Delta)$ is rainbow $K_{1,\Delta}$-free.
\end{lemma}
\begin{proof}
For each vertex $u\in C$ (resp. $v\in L$), each rainbow star with center $u$ has
at most $|\mathcal{S}_u|+d_{\overrightarrow{D}}^-(u)$ (resp. $d_{\overrightarrow{D}}^-(v)$) leaves.
By the definition of $\mathfrak{A}(n,\Delta)$ we have $|\mathcal{S}_u|+d_{\overrightarrow{D}}^-(u)\leq \Delta-1$ and $d_{\overrightarrow{D}}^-(v)\leq \Delta-1$, and hence $\mathcal{S}$ is rainbow $K_{1,\Delta}$-free.
\end{proof}

\begin{lemma}\label{lem}
For $n \ge \Delta + 1 \ge 3$, $\mathfrak{A}(n,\Delta)\neq \emptyset$.
\end{lemma}
\begin{proof}
We only need to construct a collection $\mathcal{S}\in \mathfrak{A}(n,\Delta)$ for each $n\geq \Delta+1$.
Let $C=\{x_1,x_2,\ldots,x_p\}$ and $L=\{y_1,y_2,\ldots,y_q\}$, where by definition $n = p+q$.

{\bf Case 1} $n\geq 2\Delta-1$ and $k_2\leq \Delta$.

Let $p=a(\Delta-1)+k_1$ and $q=a\Delta+b-k_1.$
For each vertex $x_i$ of $C$, let $\mathcal{T}_i$ be the (multi)set of $\Delta-1$ copies of $$x_i\vee\{y_{(i-1)\Delta+1},y_{(i-1)\Delta+2},\ldots,y_{i\Delta}\},$$
where the subscripts are taken modulo $q$. Then we define $\mathcal{S}$ to be the union of $\mathcal{T}_1,\mathcal{T}_2,\ldots,\mathcal{T}_p$.
Now we only need to show that $\mathcal{S}$ satisfies the third condition of $(i)$.
It is straightforward to verify that $d_{\overrightarrow{D}}^-(y_i)$ and $d_{\overrightarrow{D}}^-(y_j)$ differs by at most one for any $i,j\in[q]$.
Hence for each $i\in[q]$ we have
\begin{align*}
d_{\overrightarrow{D}}^-(y_i) \leq \left\lceil \frac{p\Delta}{q}\right\rceil &=\left\lceil\frac {\Delta\left(a(\Delta-1)+k_1\right)
}{a\Delta+b-k_1}\right\rceil\\
&=\left\lceil\frac{k_1(2\Delta-1)-b(\Delta-1)}
{a\Delta+b-k_1}\right\rceil + \Delta-1 \\
& \leq \Delta-1,
\end{align*}
where by assumption $b(\Delta-1)=k_1(2\Delta-1)+k_2$.

{\bf Case 2} $n\geq 2\Delta-1$ and $k_2\geq\Delta$.

Let $p=a(\Delta-1)+k_1+1$ and  $q=a\Delta+b-k_1-1.$
Note that $2\Delta-1-k_2\leq \Delta-1$. We construct $\mathcal{S}$ as follows (the subscripts are taken modulo $q$).
 \begin{itemize}
   \item Let $\mathcal{T}_1$ be the set of $\Delta-1$ copies of the star $K_{1,\Delta}$ with center $x_1$ and the set of leaves being $\{y_1,\ldots,y_{k_2-\Delta+1},x_2,\ldots,x_{2\Delta-k_2}\}$.
   \item For $2\leq i\leq 2\Delta-k_2$, let $\mathcal{T}_i$ be the set of $\Delta-2$ copies of the star $K_{1,\Delta}$ with center $x_i$ and the set of leaves being $\{y_{k_2+(i-2)\Delta+2},\ldots,y_{k_2+(i-1)\Delta+1}\}$.
   \item For $2\Delta-k_2+1\leq i\leq p$, let $\mathcal{T}_i$ be the set of $\Delta-1$ copies of the star $K_{1,\Delta}$ with center $x_i$ and the set of leaves being $\{y_{k_2+(i-2)\Delta+2},\ldots,y_{k_2+(i-1)\Delta+1}\}$.
 \end{itemize}
Now it is easy to verify that $\mathcal{S}$ satisfies the first and second condition of $(ii)$. Then it remains to prove $d_{\overrightarrow{D}}^-(y_i) \le \Delta-1$ for each $i \in [q]$.
Similar to Case 1, one can verify that $d_{\overrightarrow{D}}^-(y_i)$ and $d_{\overrightarrow{D}}^-(y_j)$ differs by at most one for any $i,j\in[q]$.
Hence for each $i\in[q]$ we have
\begin{align*}
d_{\overrightarrow{D}}^-(y_i) \leq \left\lceil \frac{p\Delta-(2\Delta-1-k_2)}{q}\right\rceil
 = \Delta-1.
\end{align*}

{\bf Case 3}  $\Delta+1\leq n\leq 2\Delta-2$.

Let $p=\left\lfloor\frac{n-1}{2}\right\rfloor$ when $n$ is odd, and $p=\left\lfloor\frac{n-1}{2}\right\rfloor$ or $\left\lfloor\frac{n+1}{2}\right\rfloor$ when $n$ is even.
For each vertex $x_i$ of $C$, let $\mathcal{T}_i$ be the set of $q-1$ copies of $$x_i\vee (L\cup \{x_{i+1},x_{i+2},\ldots,x_{i+\Delta-q}\})$$
(the subscripts are taken modulo $q$).
Let $\mathcal{S}=\mathcal{T}_1\cup \mathcal{T}_2\cup\cdots \cup\mathcal{T}_{|C|}$.
Then it is not difficult to verify $\mathcal{S}$ satisfies all the conditions in $(iv)$ since $n=p+q$.
\end{proof}

{\bf Proof of Theorem \ref{main-0}:}
Assume that $\mathcal{S}$ is a collection of stars $K_{1,\Delta}$ on the vertex set $V$ and $\mathcal{S}$ is rainbow $K_{1,\Delta}$-free.
Then $|\mathcal{S}_u|\leq \Delta-1$ for each $u\in C$ and $d_{\overrightarrow{D}}^-(v)\leq \Delta-1$ for each $v\in L$, respectively.
For simplicity of notation, we set $\overrightarrow{H}=\overrightarrow{D}[C]$ and $d_u=d_{\overrightarrow{H}}^-(u)=d^-_{\overrightarrow{D}}(u)$ for all $u\in C$.
Then there is a rainbow star $K_{1,d_u}$ centered at $u$ with $N_{\overrightarrow{H}}^-(u)$ being the set of leaves.
Moreover, no edge of the rainbow star belongs to $\bigcup_{S\in \mathcal{S}_u}E(S)$. This observation yields the following claim.
\begin{claim}\label{clm-star-1}
For each $u\in C$, $|\mathcal{S}_u|\leq\Delta-d_u-1$.
\end{claim}
\begin{proof}
For each $S\in \mathcal{S}_u$, let $S'=S-N_{\overrightarrow{H}}^-(u)$.
Then $\{S':S\in \mathcal{S}_u\}$ is a collection of stars on vertex set $V-N_{\overrightarrow{H}}^-(u)$.
Since each $S'$ has at least $\Delta-d_u$ leaves, $\{S':S\in \mathcal{S}_u\}$ contains a rainbow star $K_{1,z}$, where $z=\min\{|\mathcal{S}_u|,\Delta-d_u\}$.
Then there exists a rainbow star $K_{1,d_u+z}$ with center $u$ since
the set of the colors appearing in $K_{1,z}$ and the set of the colors appearing in $K_{1,d_u}$ are disjoint.
If $|\mathcal{S}_u|\geq\Delta-d_u$, then $d_u+z=\Delta$, a contradiction.
Hence $|\mathcal{S}_u|\leq\Delta-d_u-1$.
\end{proof}
By Claim \ref{clm-star-1}, we have
\begin{align}\label{eq-1}
\sum_{u\in C}(\Delta-d_u-1)\geq \sum_{u\in C}|\mathcal{S}_u|= |\mathcal{S}|.
\end{align}
Note that for each vertex $u\in C$,
\begin{equation}\label{eq-du-Delta}
\Delta\leq d_{\overrightarrow{D}}^+(u)=d_{\overrightarrow{H}}^+(u)+ |N_{\overrightarrow{D}}^+(u)\cap L|.
\end{equation}
Then we have
\begin{equation}\label{eq-1-1}
\begin{split}
\sum_{u\in C}|N_{\overrightarrow{D}}^+(u)\cap L|&=\sum_{u\in C}(d^+_{\overrightarrow{D}}(u)-d_{\overrightarrow{H}}^+(u))\\
&\geq \sum_{u\in C}(\Delta-d_{\overrightarrow{H}}^+(u))\\
&=|C|\Delta-\sum_{u\in C}d_{\overrightarrow{H}}^+(u)\\
&=|C|\Delta-\sum_{u\in C}d_{\overrightarrow{H}}^-(u)=\sum_{u\in C}(\Delta-d_u).
\end{split}
\end{equation}
By the rainbow $K_{1,\Delta}$-freeness of $\mathcal{S}$, $d_{\overrightarrow{D}}^-(v)=|N_{\overrightarrow{D}}^-(v)\cap C|\leq \Delta-1$ for any $v\in L$. Then
\begin{equation}\label{eq-2}
\begin{split}
(\Delta-1)|L| \ge \sum_{v\in L}|N_{\overrightarrow{D}}^-(v)\cap C|=\sum_{u\in C}|N_{\overrightarrow{D}}^+(u)\cap L|
\end{split}
\end{equation}
Combining inequalities (\ref{eq-1}), (\ref{eq-1-1}) and (\ref{eq-2}),
\begin{equation}\label{eq-final-0}
\begin{split}
|L|(\Delta-1)&\geq \sum_{u\in C}|N_{\overrightarrow{D}}^+(u)\cap L| \\
&\ge\sum_{u\in C}(\Delta-d_u)\\
&=\sum_{u\in C}(\Delta-d_u-1)+|C|\geq|\mathcal{S}|+|C|.
\end{split}
\end{equation}
Furthermore, since $|L|=n-|C|$, it follows that
\begin{align}\label{ineq-00}
|\mathcal{S}|\leq n(\Delta-1)-|C|\Delta.
\end{align}

We continue our proof by considering the following two cases.

{\bf Case 1:} $a\geq 1$.

Suppose to the contrary that
\begin{equation}\label{eq-exgraph-L1}
|\mathcal{S}|\geq a(\Delta-1)^2+k_1(\Delta-1)+\mu,
\end{equation}
where
$$\mu=\left\{
  \begin{array}{ll}
    1, &  \hbox{ if } 0\leq k_2\leq \Delta; \\
    k_2-\Delta+1, & \hbox{ if } \Delta \le k_2\leq 2\Delta-2.
  \end{array}
\right.$$
Since $|\mathcal{S}_u|\leq \Delta-1$ for each $u\in C$, it follows that
\begin{equation}\label{eq-bigc}
|C|\geq \left\lceil\frac{|\mathcal{S}|}{\Delta-1}\right\rceil \geq a(\Delta-1)+k_1+1.
\end{equation}
Recall that $n=a(2\Delta-1)+b$, it follows from inequality (\ref{ineq-00}) that
\begin{equation}\label{eq-exgraph-1}
\begin{split}|\mathcal{S}|&\leq n(\Delta-1)-|C|\Delta \\
&\leq (\Delta-1)(a(2\Delta-1)+b)-\Delta\left(a(\Delta-1)+k_1+1\right)\\
&=a(\Delta-1)^2+b(\Delta-1)-(k_1+1)\Delta.
\end{split}
\end{equation}
Combining inequalities \eqref{eq-exgraph-L1} and \eqref{eq-exgraph-1}, we obtain
\begin{equation*}\label{eq-exgraph-2}
\begin{split}
a(\Delta-1)^2+b(\Delta-1)-(k_1+1)\Delta \geq a(\Delta-1)^2+k_1(\Delta-1)+\mu.
\end{split}
\end{equation*}
which yields
\begin{align*}
k_1(2\Delta-1)+\mu \le b(\Delta-1)-\Delta = k_1(2\Delta-1)+k_2-\Delta.
\end{align*}
Then we have $k_2-\Delta\geq \mu$, a contradiction.
Therefore, we can get the upper bound of $|\mathcal{S}|$:
\begin{align*}
|\mathcal{S}| &\leq a(\Delta-1)^2+k_1(\Delta-1)+\mu-1\\
&=\left\{
\begin{array}{ll}
a(\Delta-1)^2+k_1(\Delta-1), &  \mbox{if }a\geq 1\mbox{ and }0\leq k_2\leq\Delta; \vspace{0.05cm}\\
a(\Delta-1)^2+k_1(\Delta-1)+k_2-\Delta,
&  \mbox{if }\Delta\leq k_2\leq 2\Delta-2,\\
\end{array}
\right.
\end{align*}
as stated in Theorem \ref{main-0}.

In Lemma \ref{lemm} and \ref{lem} we have shown that these bounds are tight, and we now  characterize all the collections $\mathcal{S}$ with $|S|$ attaining the maximum values.

\begin{claim}\label{clm-star-2}
The extremal cases of $\mathcal{S}$ satisfies the following properties:
\begin{enumerate}
  \item if $\Delta<k_2\leq 2\Delta-2$, then $|C|=a(\Delta-1)+k_1 +1$;
  \item if $k_2=\Delta$, then  either $|C|=a(\Delta-1)+k_1 +1$ or $|C|=a(\Delta-1)+k_1$;
  \item if $0\leq k_2\leq \Delta-1$, then $|C|=a(\Delta-1)+k_1$.
\end{enumerate}
\end{claim}
\begin{proof}
Since $|\mathcal{S}_u|\leq \Delta-1$, $|\mathcal{S}|\leq |C|(\Delta-1)$ and $|C|\geq |\mathcal{S}|/(\Delta-1)$.
Hence $|C|\geq a(\Delta-1)
+k_1+1$ if $k_2>\Delta$, and
$|C|\geq a(\Delta-1)+k_1$ if $k_2\leq \Delta$.

If $k_2\geq\Delta$, then by inequality (\ref{ineq-00}) we have
\begin{align*}
|C|\Delta&\leq n(\Delta-1)-|\mathcal{S}|=\Delta(a(\Delta-1)
+k_1+1).
\end{align*}
and $|C|\leq a(\Delta-1)+k_1+1$.
Therefore, $|C|=a(\Delta-1)
+k_1+1$ if $k_2>\Delta$, and
either $|C|=a(\Delta-1)+k_1 +1$ or $|C|=a(\Delta-1)+k_1$ if $k_2=\Delta$.
The first two statements hold.

We now prove that last statement. Suppose that $|C|>a(\Delta-1)+k_1$ as in inequality \eqref{eq-bigc}. Then by inequality \eqref{eq-exgraph-1}
$$a(\Delta-1)^2
+k_1(\Delta-1) = |\mathcal{S}| \le a(\Delta-1)^2+b(\Delta-1)-(k_1+1)\Delta,$$
which implies $k_2 \ge \Delta$, a contradiction.
\end{proof}

By Claim \ref{clm-star-2}, $|C|=a(\Delta-1)+k_1$ or $|C|=a(\Delta-1)+k_1+1$.
For the case $|C|=a(\Delta-1)+k_1$, we have $k_2\leq \Delta$.
It is easy to compute that $|C|\Delta=|\mathcal{S}|+|C|$.
Since $|\mathcal{S}_u|\leq \Delta-1$ and $|\mathcal{S}|=\sum_{u\in C}|\mathcal{S}_u|$, it follows that $|\mathcal{S}_u|=\Delta-1$ for each $u\in C$.
By inequality (\ref{eq-final-0}), we have $\sum_{u\in C}(\Delta-d_u)\geq |\mathcal{S}|+|C|$, and hence  $d_u=0$ for each $u\in C$.
It follows from the hypothesis that $\mathcal{S}$ is rainbow $K_{1,\Delta}$-free that $d_{\overrightarrow{D}}^-(v)\leq \Delta-1$ for each $v\in L$. Therefore, $\mathcal{S}$ satisfies all conditions of $(i)$, and hence $\mathcal{S}\in \mathfrak{A}(n,\Delta)$.

For the case  $|C|=a(\Delta-1)+k_1+1$, we have $k_2\geq \Delta$ and  $|L|(\Delta-1)=|\mathcal{S}|+|C|$. Hence all equalities of \eqref{eq-1}--(\ref{eq-final-0}) hold, which implies that $d_{\overrightarrow{D}}^+(u)=\Delta$ for each $u\in C$ and
$$\sum_{u\in C}d_u=|C|\Delta-(|\mathcal{S}|+|C|)=|C|(\Delta-1)-|\mathcal{S}|=2\Delta-1-k_2.$$
Therefore, $|A(\overrightarrow{D}[C])|=\sum_{u\in C}d_{\overrightarrow{H}}^-(u)=2\Delta-1-k_2.$
Note that by equalities in \eqref{eq-2} we have
$\sum_{v\in L}d_{\overrightarrow{D}}^-(v)=
|L|(\Delta-1),$
which together with $d_{\overrightarrow{D}}^-(v)\leq \Delta-1$ implies $d_{\overrightarrow{D}}^-(v)=\Delta-1$ for each $v\in L$.
Moreover, by the equality in (\ref{eq-final-0}) $|\mathcal{S}_u|=\Delta-1-d_u$ for each $u\in C$.
Recall that $d_{\overrightarrow{D}}^+(u)=\Delta$ for each $u\in C$, $\mathcal{S}_u$ consists of $\Delta-1-d_u$ copies of $K_{1,\Delta}$.
The above arguments shows that $\mathcal{S}$ satisfies all conditions of $(ii)$, and hence $\mathcal{S}\in \mathfrak{A}(n,\Delta)$.

It is worth noting that if $k_2=\Delta$, then $\mathcal{S}$ satisfies both conditions of $(i)$ and $(ii)$, which is in correspondence with
$(iii)$.

{\bf Case 2:} $a=0$.

We assume $\mathcal{S}$ contains no rainbow $K_{1,\Delta}$ and let
$$\ell^*=\frac{\sum_{u\in C}(\Delta-d_u-1)}{|C|}=\frac{\sum_{u\in C}(\Delta-d^+_{\overrightarrow{H}}(u)-1)}{|C|}.$$
By inequality (\ref{clm-star-1}),
$$|\mathcal{S}|=\sum_{u\in C}|\mathcal{S}_u|\leq\sum_{u\in C}(\Delta-d_u-1)=|C|\ell^*.$$
Then $|L|\geq \Delta-d^+_{\overrightarrow{H}}(u)$ for each vertex $u\in C$ since each vertex $u\in C$ has at least $\Delta-d^+_{\overrightarrow{H}}(u)$ out-neighbors in $L$,
It follows that $|L|\geq \ell^*+1$ and $n=|C|+|L|\geq |C|+\ell^*+1$.
Hence
$$|\mathcal{S}|\leq |C|\ell^*\leq \frac{(\ell^*+|C|)^2}{4}\leq \frac{(n-1)^2}{4}$$
and precisely $|\mathcal{S}|\leq \left\lfloor\frac{(n-1)^2}{4}\right\rfloor$.

Then we characterize all the cases in which $|\mathcal{S}|$ attains this value. When above equalities hold, we have $|C|=\ell^*=\frac{n-1}{2}$ when $n$ is odd and $\{|C|,\ell^*\}=\{\left\lfloor\frac{n-1}{2}\right\rfloor,\left\lceil\frac{n-1}{2}\right\rceil\}$
when $n$ is even.
Furthermore, $n=|C|+|L|=|C|+\ell^*+1$, which implies $|L|=\ell^*+1$.
Hence, if $n$ is odd, then $|C|=\left\lfloor\frac{n-1}{2}\right\rfloor$ and $|L|=\left\lceil\frac{n+1}{2}\right\rceil$; if $n$ is even, then either $|C|=\left\lfloor\frac{n-1}{2}\right\rfloor$ and $|L|=\left\lceil\frac{n+1}{2}\right\rceil$, or $|L|=\left\lfloor\frac{n+1}{2}\right\rfloor$ and $|C|=\left\lceil\frac{n-1}{2}\right\rceil$.

For each vertex $u\in C$, since $\ell^*+1=|L|\geq \Delta-d^+_{\overrightarrow{H}}(u)$, it follows that
 $\ell^*=|L|-1\geq \Delta-d^+_{\overrightarrow{H}}(u)-1$.
On the other hand, $|C|\ell^*=\sum_{u\in C}(\Delta-d_u-1)=\sum_{u\in C}(\Delta-d^+_{\overrightarrow{H}}(u)-1)$, which implies that
$\ell^*=|L|-1=\Delta-d^+_{\overrightarrow{H}}(u)-1$ and $d^+_{\overrightarrow{H}}(u)=\Delta-|L|$ for each $u\in C$ .
Since $\Delta\leq d_D^+(u)\leq d^+_{\overrightarrow{H}}(u)+|L|=\Delta$,
we have $d_D^+(u)=\Delta$ for each $u\in C$, which means that any two stars of $\mathcal{S}_u$ have common leaves and $L$ is contained in these leaves for each $u\in C$.
It follows that $d^+_{\overrightarrow{H}}(u)=\Delta-|L|$ and $\overrightarrow{D}[C]$ is $(\Delta-|L|)$-out-regular.
Since $|\mathcal{S}_u|\leq \Delta-1-d_{\overrightarrow{H}}^-(u)$ for each $u\in C$, it follows that
$$|\mathcal{S}|=\sum_{u\in C}|\mathcal{S}_u|\leq \sum_{u\in C}(\Delta-1-d_{\overrightarrow{H}}^-(u))= |C|(\Delta-1)-\sum_{u\in C}d^+_{\overrightarrow{H}}(u)=\left\lfloor\frac{(n-1)^2}{4}\right\rfloor=|\mathcal{S}|.$$
Hence, $|\mathcal{S}_u|=\Delta-1-d_{\overrightarrow{H}}^-(u)$ for each $u\in C$, and this number is positive since $d_{\overrightarrow{H}}^-(u)\leq |C|-1\leq \left\lceil\frac{n-1}{2}\right\rceil-1<\Delta-1$.
Therefore, $\mathcal{S}$ satisfies all conditions of $(iii)$, and  $\mathcal{S}\in \mathfrak{A}(n,\Delta)$.

\section{Proof of Theorem \ref{main-1}}\label{sec-thm2}



We shall give the proof of Theorem \ref{main-1} in this section. To this end, we first introduce the definition of $\mathfrak{B}(n,m)$.
\begin{definition}\label{def-2}
Let $2\leq n\leq m$ and $V$ be a vertex set of order $n$. We denote by $\mathfrak{B}(n,m)$ the set of families of $m(n-2)/n$ trees on $V$ such that
\begin{enumerate}
  \item $V$ is partitioned into $m/n$ parts $V_1,V_2,\ldots,V_{m/n}$, where $|V_i|=n$ for all $1 \le i \le m/n$;
  \item there are exactly $n-2$ spanning trees on $V_i$ for $1 \le i \le m/n$.
\end{enumerate}
\end{definition}

Next, we prove that if $\mathcal{T}=\{T_i: i\in[t]\}$ does not contain any rainbow tree of order $n$, then $|\mathcal{T}|\leq m(n-2)/n$.
Let $F_1$ be a maximum rainbow tree in $\mathcal{T}$ and $V(F_1)=U_1$.
Then we have $|U_1|\leq n-1$.
Let $\mathcal{S}_1=\{T_i:E(T_i)\cap E(F_1)\neq \emptyset\}$. Then we have the following claim.

\begin{claim}\label{clm-tree-1}
For each $T_i\in \mathcal{T}- \mathcal{S}_1$, $V(T_i)\cap U_1=\emptyset$.
\end{claim}
\begin{proof}
Suppose to the contrary that $V(T_i)\cap U_1\neq \emptyset$ for some $T_i\in \mathcal{T}- \mathcal{S}_1$.
Since $|U_1|\leq n-1$ and $|V(T_i)|=n$, it follows that $\partial_{T_i}(U_1)\neq \emptyset$.
If we choose an edge $f$ of $\partial_{T_i}(U_1)$, then $F_1\cup f$ is a larger rainbow tree, which contradicts the maximality of $F_1$.
\end{proof}

By Claim \ref{clm-tree-1}, $\mathcal{T}_1=\mathcal{T}- \mathcal{S}_1$ is a collection of trees of order $n$ on vertex set $V_1=V-U_1$.
Let $F_2$ be a maximum rainbow tree of $\mathcal{T}_1$, $V(F_2)=U_2$ and $\mathcal{S}_2=\{T_i:E(T_i)\cap E(F_2)\neq \emptyset\}$.
Then $|U_2|\leq n-1$ and
$\mathcal{T}_2=\mathcal{T}-\mathcal{S}_1-\mathcal{S}_2$ is a collection of trees of order $n$ on vertex set $V_2=V-U_1-U_2$ by Claim \ref{clm-tree-1}.
Continuing this process until $V_t=V-\cup_{i=1}^sU_i=\emptyset$,
we will obtain a sequence of rainbow trees $F_1,F_2,\ldots,F_s$ satisfying:
\begin{enumerate}
  \item $F_i$ is a maximum rainbow tree in $\mathcal{T}_{i-1}$ for each $i\in[s]$, where we set $\mathcal{T}_0=\mathcal{T}$;
  \item for each $i\in[s]$, $V(F_i)=U_i$, $\mathcal{S}_i=\{T_j:E(T_j)\cap E(F_i)\neq \emptyset\}$, $\mathcal{T}_i=\mathcal{T}-\bigcup_{j\in[i]}\mathcal{S}_j$, $V_i=V-\bigcup_{j\in[i]}U_i$, and $\mathcal{T}_i$ is a collection of trees of order $n$ on vertex set $V_i$;
  \item we may have $|U_i|=1$ for some $i \in [s]$, in which case    $\mathcal{T}_{i-1}=\mathcal{T}_i=\emptyset$.
\end{enumerate}
Since each $\mathcal{T}_i$ is a subset of $\mathcal{T}$, $\mathcal{T}_i$ does not contain any rainbow tree of order $n$.
Hence $|V(F_i)|\leq n-1$ for each $i\in[s]$.
Note that $\mathcal{S}_1,\mathcal{S}_2, \ldots,\mathcal{S}_{s}$ forms a partition of $\mathcal{T}$.
Hence
\begin{equation}\label{eq-uit}
t = \sum_{i\in[s]}|\mathcal{S}_i|=\sum_{i\in[s]}|E(F_i)|=\sum_{i\in[s]}(|U_i|-1).
\end{equation}
Similarly, $U_1, U_2,\ldots,U_s$ forms a partition of $V$ and we have $\sum_{i\in[s]}|U_i|=|V|=m$.
Then
\begin{align}\label{ineq-2}
t=m-s.
\end{align}

To show that $t\leq m(n-2)/n$, we need to prove that $s\geq 2m/n$, the proof of which would be separated into several claims. At first,
for each $2\leq i\leq s$, let
$$\mathcal{R}_i=\{T_j\in \bigcup_{\ell\in[i-1]}\mathcal{S}_\ell: V(T_j)\cap U_i\neq \emptyset\}$$
and $|\mathcal{R}_i|=r_i$.

\begin{claim}\label{clm-tree-2}

For each $T_k\in \mathcal{R}_i$ and each edge $f\in T_\ell$ with one endpoint in $U_i$ the other endpoint in $V-U_i$, there is a rainbow tree $T^*$ such that $F_i\cup f$ is a rainbow subtree of $T^*$,  $|T^*|=|U_i|+r_i$ and the edges in $E(T^*)-E(F_i)$ come from distinct trees in $\mathcal{R}_i$. Moreover, $|U_i|+r_i\leq n-1$ for each $2\leq i\leq s$.
\end{claim}
\begin{proof}
Let $\mathcal{R}_i=\{T_{c_1},T_{c_2},\cdots,T_{c_{r_i}}\}$.
Without loss of generality, we may assume $k=c_1$ and set $f=f_1$. Then it is straightforward to verify that $F_i\cup f_1$ is a rainbow tree.
Since $\mathcal{T}$ does not contain any rainbow tree of order $n$, we have $|F_i\cup f_1|<n$. Further, there exists an edge $f_2$ of $T_{c_2}$ belonging to $\partial_{T_{c_2}}(V(F_i\cup f_1))$ since $V(T_{c_2})\cap U_i\neq \emptyset$ and $|T_{c_2}|=n>|F_i\cup f_1|$.
Then $F_i\cup \{f_1,f_2\}$ is again a rainbow tree.
Repeating this process, we will get a rainbow tree $T^* = F_i\cup \{f_1,f_2,\ldots,f_{r_i}\}$ with $f_j\in E(T_{c_j})$ for each $j\in[r_i]$.
It is easy to see that $T^*$ is a rainbow tree of order $|U_i|+r_i$ containing $F_i\cup f$ as a rainbow subtree, and $|U_i|+r_i\leq n-1$ since $\mathcal{T}$ does not contain any rainbow tree of order $n$.
\end{proof}

\begin{claim}\label{clm-tree-3}
$\sum_{2\leq i\leq s}r_i\geq \sum_{s\in[t-1]}(|U_i|-1)$.
\end{claim}
\begin{proof}
For an integer $i\in[s-1]$, let $T$ be an arbitrary tree of $\mathcal{S}_i$.
By the definition of $F_i$, we have that $V(T)\cap U_j=\emptyset$ for each $j<i$.
However, it follows from $\partial_{T}(U_i)\neq \emptyset$ that there exists an integer $\ell$ with $i< \ell\leq s$ such that there is an edge in $\partial_{T}(U_i)$ connecting $U_i$ and $U_\ell$.
Since $i\in[t-1]$ and $T\in\mathcal{S}_i$ is chosen arbitrarily,  each tree of $\bigcup_{i\in[s-1]}\mathcal{S}_i$ contributes at least one to $\sum_{2\leq i\leq t}r_i$.
Thus,
$$\sum_{2\leq i\leq s}r_i\geq \left|\bigcup_{i\in[s-1]}\mathcal{S}_i\right|= \sum_{i\in[s-1]}|\mathcal{S}_i|= \sum_{i\in[s-1]}(|U_i|-1).$$
\end{proof}

\begin{claim}\label{clm-tree-4}
$|U_s|=1$.
\end{claim}
\begin{proof}
Suppose to the contrary that $|U_s|\geq 2$, which implies $\mathcal{S}_t\neq \emptyset$.
Note that $\mathcal{S}_s$ is a collection of trees of order $n$ on vertex set $V_{s-1}$ and hence $|V_{s-1}|\geq n$. Then
$V_s=V_{t-1}-U_s$ contains at least one vertex since $|U_s|\leq n-1$. It follows that there will exists an $F_{s+1}$, a contradiction.
\end{proof}

By Claim \ref{clm-tree-2}, \ref{clm-tree-3}, \ref{clm-tree-4}, identity \eqref{eq-uit} and \eqref{ineq-2}, we have
\begin{equation}\label{inq-ex-2}
\begin{split}
2(m-s) &= 2\sum_{i\in[s]}(|U_i|-1)\\
&=\sum_{i\in[s]}(|U_i|-1)+\sum_{i\in[s-1]}(|U_i|-1)\\
&\leq
\sum_{i\in[s]}(|U_i|-1)+\sum_{2\leq i\leq s}r_i\\
&=(|U_1|-1)+\sum_{2\leq i\leq s}(|U_i|-1+r_i)\\
&\leq n-2+(s-1)(n-2)\\
&=s(n-2),
\end{split}
\end{equation}
$$2(m-t)=2\sum_{i\in[t]}(|U_i|-1)\leq t(n-2),$$
and hence
$$s\geq \frac{2m}{n}, \quad t=m-s\leq m-\frac{2m}{n}=\frac{m}{n}(n-2).$$
Hence, $|\mathcal{T}|\leq m(n-2)/n$.

We now characterize all collections $\mathcal{T}$ with $|\mathcal{T}|= m(n-2)/n$ by induction.
It is obvious that $\mathcal{T}\in \mathfrak{B}(n,m)$ if $m=n$.
Note that $|\mathcal{T}|= m(n-2)/n$ indicates that $t=2m/n$ and all equalities in Claim \ref{clm-tree-2}, \ref{clm-tree-3} and inequality (\ref{inq-ex-2}) hold. Then
\begin{align*}
|U_i|+r_i= n-1
\end{align*}
for each $2\leq i\leq s$, and
each $T\in\mathcal{S}_i$ contributes exactly one to $\sum_{2\leq i\leq s}r_i$, where $i\in[t-1]$. This implies that for each $T\in \mathcal{T}$, there exist exactly two indices $j,j'\in[t-1]$ such that $V(T)\subseteq U_j\cup U_{j'}$, $V(T)\cap U_j\neq \emptyset$ and $V(T)\cap U_{j'}\neq \emptyset$. For the case $i = s$, we have
$r_s=n-2$ since $|U_s| = 1$ by Claim \ref{clm-tree-4}.
Moreover, there are exactly $n-2$ trees of $\mathcal{T}$, say $T_{p_1},T_{p_2},\ldots,T_{p_{n-2}}$, such that $u\in V(T_{p_i})$ for each $i\in[n-2]$, where we set $U_s=\{u\}$.
Assume that $T_{p_i}\in\mathcal{S}_{q_i}$ for each $i\in[n-2]$.
Then $V(T_{p_i})\subseteq U_{q_i}\cup\{u\}$ for each $i\in[n-1]$.
Since $|T_{p_i}|=n$, it follows that $|U_{q_i}|=n-1$ and $V(T_{p_i})= U_{q_i}\cup\{u\}$ for each $i\in[n-1]$.

\begin{claim}
$q_1=q_2=\cdots=q_{n-2}$.
\end{claim}
\begin{proof}
Let
$I=\{p_i: i\in[n-2] \mbox{ and } q_i=q_1\}$ and $I'=C(F_{q_1})-I$, where $C(F_{q_1})=\{i:T_i\in \mathcal{S}_{q_1}\}$.
By contradiction, we suppose that $|I|<n-2$.
Then $I'\neq \emptyset$ since $|U_{q_i}|=n-1$. Selecting one element $k\in I'$, we have $k\notin\{p_1,p_2,\ldots,p_{n-2}\}$ since $T_k \in \mathcal{S}_{q_1}$.
Then there exists an integer $q_1 < \ell < s$ such that $V(T_k)\subseteq U_{q_1}\cup U_\ell$ and an edge $f\in E(T_k)$ joining $U_\ell$ and $U_{q_1}$, and by definition $T_{q_1}\in \mathcal{R}_{\ell}$.

Note that $F_\ell\cup f$ is a rainbow tree and recall $|U_{\ell}|+r_{\ell}=n-1$, and then
by Claim \ref{clm-tree-2} we can construct a rainbow tree $T^*$ containing $F_\ell\cup f$ such that $|T^*|=n-1$.
It follows from $V(T_{p_1})= U_{q_1}\cup\{u\}$ that $V(T_{p_1})\cap U_\ell=\emptyset$ and $T_{p_1} \notin \mathcal{R}_l$.
Then $p_1\notin C(T^*)$, where $C(T^*)=\{i\in[t]:E(T_i)\cap E(T^*)\neq\emptyset\}$.
Since $V(T^*)\cap V(T_{p_1})\neq \emptyset$ and $|T^*|<|T_{p_1}|$, it follows that $\partial_{T_{p_1}}(V(T^*))\neq \emptyset$, say $f'\in \partial_{T_{p_1}}(V(T^*))$.
Thus, $T^*\cup f'$ is a rainbow tree of order $n$, a contradiction.
\end{proof}

By the above claim, we may let $q=q_1=q_2=\ldots=q_{n-2}$. Then
$$\mathcal{S}_{q}=\{T_{p_1},T_{p_2},\ldots,T_{p_{n-2}}\}$$
since $|\mathcal{S}_q| = |U_q|-1 \le n-2$.
Hence, $F_q$ is a rainbow tree of order $n-1$ whose edges come from $T_{p_1},T_{p_2},\ldots,T_{p_{n-2}}$.
Since $u\in V(T_{p_i})$ for each $i\in[n-2]$, it follows that $T_{p_1},T_{p_2},\ldots,T_{p_{n-2}}$ are $n-2$ trees on vertex set $V'=U_q\cup \{u\}$.

Suppose that there exists some $T\in\mathcal{T}-\mathcal{S}_{q}$ and $v \in V(T)\cap V'$. Then $v \neq u$ since $\mathcal{S}_q = \mathcal{R}_s$.
It follows that $v \in U_q \cap V(T)$ and $\partial_{T}(U_q) \neq \emptyset$, from which we can construct a rainbow tree of order $n$, a contradiction. Hence $V(T)\cap V'=\emptyset$ for each $T\in\mathcal{T}-\mathcal{S}_{q}$ and $\mathcal{F}=\mathcal{T}-\mathcal{S}_{q}$ is a collection of  $(m-n)(n-2)/n$ trees of order $n$ on vertex set $V-V'$ with $|V-V'|=m-n$.
Since $\mathcal{F}$ also contains no rainbow tree of order $n$, we have $\mathcal{F}\in \mathfrak{B}(n,m-n)$ and $\mathcal{T}\in \mathfrak{B}(n,m)$ by induction.

\section{Concluding remarks}
In this paper, we mainly consider collections of two types of tree structures: stars and general trees.
It is natural to ask
how many Hamiltonian paths on vertex set $V$ ($|V|=n$) contains a rainbow Hamiltonian path.
It is clear that the minimum number of Hamiltonian paths that are needed in the this question is at least $n-1$, since a rainbow Hamiltonian path has $n-1$ edges.
In graphic matroids, a Hamiltonian path (or more generally, a spanning tree) is a basis. As mentioned in Introduction, the validness Rota's basis conjecture will imply that $n-1$ Hamiltonian paths on $V$ can be decomposed into $n-1$ rainbow spanning trees. Therefore, when we assume the conjecture is true, another natural problem is that whether this question is a corollary of Rota's basis conjecture. This surmise seems to be plausible because in this case there are at most $2(n-1)$ edges incident with each vertex of $V$, and these edges are partitioned into $n-1$ nonempty sets in the decomposition.

\section{Acknowledgements}

Ping Li is supported by the National Natural Science Foundation of China (No. 12201375). Luyi Li is also supported by the Tianjin Research Innovation Project for Postgraduate Students (No. 2022BKY039). Ethan Li is supported by the Fundamental Research Funds for the Central Universities (No. GK202207023).

\end{document}